\documentclass{amsart}

\usepackage{verbatim}
\usepackage{amsmath}
\usepackage{amsfonts}
\newtheorem{theorem}{Theorem}[section]
\newtheorem{lemma}[theorem]{Lemma}
\newtheorem{corollary}[theorem]{Corollary}
\newtheorem{proposition}[theorem]{Proposition}
\newtheorem{conjecture}[theorem]{Conjecture}

\theoremstyle{definition}

\theoremstyle{remark}
\newtheorem{remark}[theorem]{Remark}

\numberwithin{equation}{section}

\newcommand{\namedthm}[2]{\theoremstyle{plain}
   \newtheorem*{thm#1}{#1}\begin{thm#1}#2\end{thm#1}}

\def\bcw{\mathbin{\bigcirc\mkern-15mu\wedge}}

\begin{document}

\title[Conformally invariant gap theorem]{A conformally invariant gap theorem characterizing $\mathbb{CP}^2$ via the Ricci flow}

\author{Sun-Yung A. Chang}
\address{Department of Mathematics, Princeton University, Princeton, NJ 08544, USA}
\email{chang@math.princeton.edu}
\thanks{The first author was supported in part by NSF Grant MPS-1509505.}

\author{Matthew Gursky}
\address{Department of Mathematics, University of Notre Dame, Notre Dame, IN 46556, USA}
\email{Matthew.J.Gursky.1@nd.edu}
\thanks{The second author was supported in part by NSF Grant DMS-1811034.}

\author{Siyi Zhang}
\address{Department of Mathematics, Princeton University, Princeton, NJ 08544, USA}
\email{siyiz@math.princeton.edu}

\date{\today}



\begin{abstract} We extend the sphere theorem of \cite{CGY03} to give a conformally invariant characterization of $(\mathbb{CP}^2, g_{FS})$.  In particular, we introduce a conformal invariant $\beta(M^4,[g]) \geq 0$ defined on conformal four-manifolds satisfying a `positivity' condition; it follows from \cite{CGY03} that if $0 \leq \beta(M^4,[g]) < 4$, then $M^4$ is diffeomorphic to $S^4$.  Our main result of this paper is a `gap' result showing that if $b_2^{+}(M^4) > 0$ and $4 \leq \beta(M^4,[g]) < 4(1 + \epsilon)$ for $\epsilon > 0$ small enough, then $M^4$ is diffeomorphic to $\mathbb{CP}^2$.  The Ricci flow is used in a crucial way to pass from the bounds on $\beta$ to pointwise curvature information.
\end{abstract}

\maketitle



\section{Introduction}

In \cite{CGY03}, the first two authors with P. Yang proved a conformally invariant sphere theorem in dimension four.  In this paper we extend the results of \cite{CGY03} to give a characterization of complex projective space. 
To state our results we begin by establishing our notation and conventions.

If $(M^4,g)$ is a smooth, closed Riemannian four-manifold, we denote the Riemannian curvature tensor by $Rm$ (or $Rm_g$ if we need to specify the metric), the Ricci tensor by $Ric$, and the scalar curvature by $R$.  We also denote the Weyl curvature tensor by $W$, and the Schouten tensor
\begin{align} \label{Pdef}
P = \frac{1}{2} \big( Ric - \frac{1}{6}R \cdot g \big).
\end{align}
We remark that the definition of the Schouten tensor in \cite{CGY03} (denoted by $A$) differed from the formula in (\ref{Pdef}) by a factor of two; however, in this paper we adopt the more common convention.  In terms of the Weyl and Schouten tensors the Riemannian curvature tensor can be decomposed as
\begin{align} \label{rot}
Rm = W + P \bcw g
\end{align}
where $\bcw$ is the Kulkarni-Nomizu product.  There are two important consequences of this identity:  First, since the Weyl tensor is conformally invariant, it follows that the behavior of the curvature tensor under a conformal change of metric is determined by the transformation of the Schouten tensor.  The second consequence is that the splitting induces a splitting of the Euler form, so that the Chern-Gauss-Bonnet formula can be expressed as
\begin{align} \label{CGB}
8 \pi^2 \chi(M) = \int \|W \|^2 \,dv + 4 \int \sigma_2(g^{-1} P) \,dv,
\end{align}
where \\

$\bullet$ $\| \cdot \|$ denotes the norm of the Weyl tensor, viewed as an endormorphism of $\Omega^2(M)$, the bundle of two-forms.  Note that this differs from the norm of Weyl when viewed as a four-tensor, and the two norms are related by
\begin{align*}
\| W \| = \frac{1}{4} | W|^2.
\end{align*}

$\bullet$  $g^{-1} P$ denotes the $(1,1)$-tensor (interpreted as an endomorphism of the tangent space at each point) obtained by `raising an index' of the Schouten tensor, and $\sigma_2(g^{-1}P)$ is the second elementary symmetric polynomial applied to its eigenvalues. To simplify notation we will henceforth write $\sigma_2(P)$ in place of $\sigma_2(g^{-1}P)$. \\

It follows from the conformal invariance of the Weyl tensor that both integrals in (\ref{CGB}) are conformally invariant.  While their sum is a topological invariant, their ratio can be arbitrary.  As we now explain, when the scalar curvature is positive the ratio does carry geometric and topological information.   

Given a Riemannian manifold $(M^n, g)$ of dimension $n \geq 3$, let $[g]$ denote the equivalence class of metrics pointwise conformal to $g$, and $Y(M^n,[g])$ denote the Yamabe invariant:
\begin{align*}
Y(M^n,[g]) = \inf_{ \tilde{g} \in [g] } Vol(\tilde{g})^{-\tfrac{n-2}{n}} \int R_{\tilde{g}}\,dv_{\tilde{g}}.
\end{align*}
We can also express the Yamabe invariant in terms of the first symmetric function of the Schouten tensor: it follows from (\ref{Pdef}) that
\begin{align*}
\sigma_1(P) = \frac{R}{2(n-1)},
\end{align*}
hence
\begin{align*}
Y(M^n,[g]) = \inf_{ \tilde{g} \in [g] } 2 (n-1) Vol(\tilde{g})^{-\tfrac{n-2}{n}} \int \sigma_1(P_{\tilde{g}}) \,dv_{\tilde{g}}.
\end{align*}
With this interpretation of the Yamabe invariant, in dimension four we should view the conformal invariant $\int \sigma_2(P) \,dv$ as a kind of ``second Yamabe invariant'' (see \cite{GLW04}, \cite{Sheng08}, \cite{GLW10}).  We therefore define
\begin{align} \label{Y1}
\mathcal{Y}_1^{+}(M^4) = \{ g\ :\ Y(M^4,[g]) > 0 \},
\end{align}
and
\begin{align} \label{Y2def}
\mathcal{Y}_2^{+}(M^4) = \{ g \in \mathcal{Y}_1^{+}(M)\ :\ \int \sigma_2(P_g)\,dv_g > 0 \}.
\end{align}

By a classical result of Lichnerowicz, there are topological obstructions to $\mathcal{Y}_1(M^4)$ being non-empty \cite{Lich}.  There are also topological implications of $\mathcal{Y}_2(M^4)$ being non-empty:  by \cite{Gur98}, if $\mathcal{Y}_2^{+}(M^4) \neq \emptyset$ then the first Betti number $b_1(M^4) = 0$.  In fact, it follows from \cite{CGY02} that $[g]$ contains a metric $\tilde{g}$ with positive Ricci curvature.

Returning to the Chern-Gauss-Bonnet formula, for metrics $g \in \mathcal{Y}_2^{+}(M^4)$ we define the conformal invariant
\begin{align} \label{betadef}
\beta(M^4,[g]) = \dfrac{ \int \| W_g \|^2 \,dV_g }{ \int \sigma_2( P_g) \, dv_g } \geq 0.
\end{align}
We also define smooth invariant
\begin{align} \label{betaM}
\beta(M^4) = \inf_{[g]} \beta(M^4,[g]).
\end{align}
If $\mathcal{Y}_2^{+}(M^4) = \emptyset$, we set $\beta(M^4) = -\infty$.

The main results of \cite{CGY03} give a (sharp) range for $\beta$ that imply the underlying manifold is the sphere:

\namedthm{Theorem 1} {Suppose $M^4$ is oriented.  If $g \in \mathcal{Y}_2^{+}(M^4)$ with
\begin{align} \label{B1}
\beta(M^4,[g]) < 4,
\end{align}
then $M^4$ is diffeomorphic to $S^4$.  In particular, if $M^4$ satisfies
\begin{align*}
-\infty < \beta(M^4) < 4,
\end{align*}
then the same conclusion holds.

Furthermore, if $M^4$ admits a metric with $\beta(M^4,[g]) = 4$, then one of the following must hold:

\begin{itemize}

\item $M^4$ is diffeomorphic to $S^4$; or \\

\item $M^4$ is diffeomorphic to $\mathbb{CP}^2$ and $g \in [g_{FS}]$, where $g_{FS}$ denotes the Fubini-Study metric.

\end{itemize}

}

As a corollary we have the following characterization of manifolds for which $\beta(M^4) = 0$:

\namedthm{Corollary 1} {Assume $M^4$ is oriented.  Then $\beta(M^4) = 0$ if and only if $M^4$ is diffeomorphic to $S^4$.  Furthermore, $\beta(M^4,[g]) = 0$ if and only if $g \in [g_0]$, where $g_0$ denotes the round metric.   }

\vskip.2in

\noindent {\bf Remarks.}  \begin{enumerate}

\item For the case of equality, we note that if $\beta(M^4,[g]) = 0$, then the Weyl tensor $W_{g} \equiv 0$ and it follows that $(M^4,g)$ is locally conformally flat.  By our observations above, since $[g]$ admits a metric with positive Ricci curvature, by Kuiper's theorem \cite{Kuiper} $(M^4,g)$ is conformally equivalent to $(S^4,g_0)$ or $(\mathbb{RP}^4, g_0)$, where $g_0$ is the standard metric.  \\

\item  There are a number of other sphere-type theorems under integral curvature conditions; see for example \cite{CD10}, \cite{CNd10}, \cite{GLW10}, \cite{CZ14}, \cite{BC15}, \cite{L16},
\cite{BC17}, and the references in \cite{CGY03}.  \\

\end{enumerate}

Our first goal in this paper is initiating the study of four-manifolds with
\[ \beta(M^4,[g]) \geq 4. \]
 Suppose $M^4$ is oriented, and let $b_2(M^4)$ denote the second Betti number.  Then we can write $b_2 = b_2^{+} + b_2^{-}$, where $b_2^{\pm}$ denotes the dimension of the space of self-dual/anti-self-dual harmonic two-forms.  If $b_2(M^4) \neq 0$, then by changing the orientation if necessary we may assume $b_2^{+} > 0$.

\namedthm{Theorem A} {Suppose $M^4$ is oriented and $b_2^{+}(M^4) > 0$.  There is an $\epsilon > 0$ such that if $M^4$ admits a metric $g \in \mathcal{Y}_2^{+}(M^4)$ with
\begin{align} \label{Bep}
4 \leq \beta(M^4,[g]) < 4(1 + \epsilon),
\end{align}
then $M^4$ is diffeomorphic to $\mathbb{CP}^2$.

}

This `gap' theorem immediately gives a characterization of manifolds with $b_2(M^4) \neq 0$ and $\beta(M^4) = 4$: \\

\namedthm{Theorem B} {Suppose $M^4$ is oriented and $b_2^{+}(M^4) > 0$.  If
\begin{align} \label{B2}
\beta(M^4) = 4,
\end{align}
then $M^4$ is diffeomorphic to $\mathbb{CP}^2$.  Moreover, $\beta(M^4,[g]) = 4$ if and only if $g \in [g_{FS}]$.

}

The proof of Theorem A is similar in approach to the proof of Theorem 1 in \cite{CGY03}.  The first step is to find a conformal representative satisfying a pointwise curvature condition that encodes the integral assumptions of the theorem.  In \cite{CGY03} this involved solving a modified version of the $\sigma_2$-Yamabe problem.  However, in the proof of Theorem B it is more natural to consider a modified version of the Yamabe problem introduced by the second author in \cite{Gur00}.  In particular we show that a metric satisfying the assumptions of Theorem B can be conformally deformed to a metric that is ``almost self-dual Einstein'' in an $L^2$-sense, and whose scalar curvature satisfies a condition similar to the condition satisfied by the scalar curvature of a K\"ahler metric.

As in the proof of Theorem 1, the second step involves the Ricci flow.  In \cite{CGY03} the weak pinching result of Margerin \cite{Mar98} played a crucial role.
To prove Theorem B, we show that the Ricci flow with the conformal representative constructed in the first step as the initial metric, will have uniform bounds on the curvature and the Sobolev constant on a fixed time interval $[0, T_0]$, with $T_0 >0$ depending on the pinching constant.
These estimates together with the convergence theory of Cheeger-Gromov-Taylor \cite{CGT} imply the family of Ricci flows $g_{j}(T_0)$ (up to a subsequence) will converge to the Fubini-Study metric.

In view of Theorems A and B, we make several conjectures.  The first is that Theorem B remains valid if we drop the assumption on $b_2^{+}(M^4)$:

\begin{conjecture} \label{Con1} If
\begin{align} \label{B2}
\beta(M^4) = 4,
\end{align}
then $M^4$ is diffeomorphic to $\pm \mathbb{CP}^2$.  Moreover, $\beta(M^4,[g]) = 4$ if and only if $g \in [g_{FS}]$.
\end{conjecture}

It is clear that $4$ is a `special' or `critical' value of $\beta$, at which the topology of the underlying manifold can change.  A natural question is the next critical value.  As a corollary of \cite{Gur98}, we have the following estimate for manifolds with indefinite intersection form, i.e., $b_2^{+}(M^4), b_2^{-}(M^4) > 0$.  \\

\namedthm{Theorem C} {Suppose $M^4$ is oriented, $b_2(M^4) \neq 0$, and the intersection form of $M^4$ is indefinite.  If $\mathcal{Y}_2^{+}(M^4)$ is non-empty, then
\begin{align} \label{B3}
\beta(M^4) \geq 8.
\end{align}
Moreover, if $M^4$ admits a metric with $\beta(M^4,[g]) = 8$, then $g \in [g_p]$, where $g_p = g_{S^2} \oplus g_{S^2}$ is the product metric on $S^2 \times S^2$.}

Our next conjecture is that we can weaken the condition $\beta(M^4) = 4$, and characterize the possible topological types of manifolds admitting metrics with $\beta$ between $4$ and $8$:

\begin{conjecture} \label{Con2}  If $M^4$ is oriented and admits a metric $g \in \mathcal{Y}_2^{+}(M^4)$ with
\begin{align} \label{less8}
0 \leq \beta(M^4,[g]) < 8,
\end{align}
then $M^4$ is diffeomorphic to $S^4$ or $\pm \mathbb{CP}^2$.
\end{conjecture}

\vskip.2in



\section{Preliminaries} \label{Prelim}

In this section, we state and prove some preliminary results, including the proof of Theorem C.  Several of the results in this section are based on the following result in \cite{Gur98}:

\begin{theorem} \label{ThmGur98}
Let $M^4$ be a closed oriented four-manifold with $b_2^{+}(M^4) > 0$. Then for any metric $g$ with $Y(M^4,[g])\geq0$,
\begin{equation}\label{Gursky theorem}
  \int_M||W^+||^2dv\geq\frac{4\pi^2}{3}(2\chi(M^4)+3\tau(M^4))
\end{equation}
where $\chi(M^4)$ and $\tau(M^4)$ denote the Euler characteristic and signature of $M^4$, respectively.

Furthermore:
\begin{enumerate}
  \item Equality is achieved in (\ref{Gursky theorem}) by some metric $g$ with $Y(M^4,[g])>0$ if and only if $g$ is conformal to a (positive) K\"ahler-Einstein metric $g_{KE}=e^{2w}g$.
  \item Equality is achieved in (\ref{Gursky theorem}) by some metric $g$ with $Y(M,[g])=0$ if and only if $g$ is conformal to a Ricci-flat anti-self-dual K\"ahler-Einstein metric $g_{KE}=e^{2w}g$.
\end{enumerate}
\end{theorem}

The following lemma is the first application of this result:

\begin{lemma} \label{Lemma21}
Let $M^4$ be a closed, oriented four-manifold admitting a metric $g \in \mathcal{Y}_2^{+}(M^4)$ with
\begin{equation}\label{Borderline comparison}
 \beta(M^4,[g]) < 8.
\end{equation}
If $b_2^+(M^4)>0$, then the signature of $M^4$ satisfies
\[\tau(M^4)>0.\]
\end{lemma}

\begin{proof}
The signature formula implies
\begin{equation}\label{W and signature}
  \int ||W_g||^2 \,dv_g =\int_M||W^{+}_g||^2\,dv_g + \int ||W^{-}_g ||^2 \,dv_g = 2\int ||W_g^{+}||^2 \,dv_g -12\pi^2\tau(M^4).
\end{equation}
By (\ref{Borderline comparison}),
\begin{align*}
\int \sigma_2(P_g) \,dv_g > \frac{1}{8} \int \| W_g \|^2 \,dv_g.
\end{align*}
Substituting this into the Chern-Gauss-Bonnet formula, we have
\begin{equation}\label{W and Euler}
  8\pi^2\chi(M^4)= \int ||W_g ||^2 \,dv_g + 4 \int \sigma_2(P_g) \,dv_g > \frac{3}{2} \int ||W_g||^2 \,dv_g.
\end{equation}
Combining (\ref{W and signature}) and (\ref{W and Euler}), we get
\begin{equation}\label{simplification}
8\pi^2 \chi(M^4)>3\int  ||W^{+}_g||^2 \,dv_g - 18\pi^2\tau(M^4),
\end{equation}
and this inequality can be rewritten as
\begin{equation}\label{rewritten}
  \int ||W_g^{+}||^2dv_g < \frac{4}{3}\pi^2(2\chi(M^4)+3\tau(M^4))+2\pi^2\tau(M^4).
\end{equation}
Since $b_2^+(M^4)>0$ and $Y(M^4,[g])>0$, (\ref{Gursky theorem}) in Theorem 2.1 implies
\begin{equation}\label{Gursky estimate}
  \int ||W_g^{+}||^2 \,dv_g \geq\frac{4}{3}\pi^2(2\chi(M^4)+3\tau(M^4)).
\end{equation}
Therefore, combining (\ref{rewritten}) and (\ref{Gursky estimate}), we conclude
\[\tau(M^4)>0.\]
\end{proof}

\begin{remark}
This lemma is sharp in the following sense:
Suppose $(M,g)$ is isometric to $(S^2\times{S^2},g_{prod})$. In this case, $b_2^+(M^4)=b_2^-(M^4)=1$, $\tau(M^4)=0$ and
\begin{equation}\label{sigma_2 on S2S2}
\int ||W_{g} ||^2 \,dv_g= 8 \int \sigma_2(P_g) \,dv_g = \frac{64}{3}\pi^2.
\end{equation}
\end{remark}

\begin{corollary} \label{Cor21}
Let $M^4$ be a closed, oriented four-manifold admitting a metric $g \in \mathcal{Y}_2^{+}(M^4)$ with
\begin{equation} \label{corollary 2.4}
\beta(M^4,[g]) < 8.
\end{equation}
Then either $b_2(M^4)=0$, or the intersection form is definite.
\end{corollary}

\begin{proof}
Suppose $b_2(M^4) \neq 0$ and $b_2^{+}(M^4) \cdot b_2^{-}(M^4) >0$.  Then Lemma \ref{Lemma21} implies $\tau(M^4)=b_2^+(M^4)-b_2^-(M^4)>0$.  Since $b_2^-(M^4)$ is also non-zero, we can apply Lemma \ref{Lemma21} to $M^4$ endowed with the opposite orientation to show that the signature is again positive.  This is a contradiction, since changing the orientation changes the sign of the signature.  It follows that $b_2^{+}(M^4) \cdot b_2^{-}(M^4) = 0$, hence the intersection form is definite.
\end{proof}

Combining the two previous results with an {\em a priori} upper bound for the total $\sigma_2$-curvature, we can prove the following:

\begin{lemma}\label{Lemma25}
Let $M^4$ be a closed, oriented four-manifold admitting a metric $g \in \mathcal{Y}_2^{+}(M^4)$ with
\begin{equation} \label{8again}
\beta(M^4,[g]) < 8.
\end{equation}
If $b_2(M^4) > 0$, then (after possibly changing the orientation) $b_2^{+}(M^4) = 1$ and $b_2^{-}(M^4) = 0$.
\end{lemma}

\begin{proof}
By Corollary \ref{Cor21} we may choose an orientation for which the intersection form is positive definite, so $b_2^{+}(M^4) > 0$ and $b_2^{-}(M^4) = 0$.  Also, by Corollary F of \cite{Gur98}, if $g\in\mathcal{Y}_2^+(M^4)$, $b_1(M^4) = 0$.  Therefore,
\begin{align} \label{chars} \begin{split}
\chi(M^4) &= 2 + b_2^{+}(M^4), \\
\tau(M^4) &= b_2^{+}(M^4) > 0.
\end{split}
\end{align}
By the Chern-Gauss-Bonnet formula and (\ref{corollary 2.4}),
\begin{align} \label{CG} \begin{split}
8\pi^2\chi(M^4) &= \int \| W_g \|^2 \,dv_g + 4 \int \sigma_2(P_g) \,dv_g \\
&< 12 \int \sigma_2(P_g) \,dv_g.
\end{split}
\end{align}
By Theorem B of \cite{Gur99}, we have the bound
\begin{equation} \label{Gursky's inequality}
\int \sigma_2(P_g) \,dv_g \leq 4 \pi^2,
\end{equation}
and equality holds if and only if $(M^4,g)$ is conformally equivalent to the round sphere.  More generally, since the integral is conformally invariant it is easy to show that
\begin{align} \label{S2Y}
\int \sigma_2(P_g) \,dv_g \leq \frac{1}{96} Y(M^4,[g])^2.
\end{align}

Since $b_2(M^4) > 0$ strict inequality must
hold, and substituting this into (\ref{CG}) we get
\begin{align} \label{CG2} \begin{split}
8\pi^2\chi(M^4) &< 12 \int \sigma_2(P_g) \,dv_g  \\
&< 48 \pi^2,
\end{split}
\end{align}
hence $\chi(M^4)<6$. By (\ref{chars}), we see that $1 \leq b_2^+(M^4) \leq 3$.
It therefore suffices to rule out the possibilities $b_2^+(M^4)=2$ and $b_2^+(M^4)=3$.

If $b_2^+(M^4)=2$ then $\chi(M^4)=4$, so by the Chern-Gauss-Bonnet formula
\[\int ||W_g ||^2 \,dv_g +  4 \int \sigma_2(P_g) \,dv_g = 32\pi^2.  \]
Also, $b_2^+(M^4)=2$ implies $\tau(M^4) = 2$, so the signature formula gives
\begin{align*}
\int ||W^{+}_g ||^2 \,dv_g &= \int ||W^{-}_g ||^2 \,dv_g + 24\pi^2 \\
&\geq 24 \pi^2.
\end{align*}
It follows that
\[\int ||W_g||^2 \,dv_g \geq 24\pi^2, \quad \int \sigma_2(P_g) \,dv_g \leq  2 \pi^2.\]
Therefore,
\[\int ||W_g ||^2 \,dv_g \geq 12 \int \sigma_2(P_g) \,dv_g,\]
which contradicts (\ref{8again}).

If $b_2^+(M^4)=3$, we can apply the same argument to conclude
\[\int ||W_g ||^2 \,dv_g \geq 36 \int \sigma_2(P_g) \,dv_g,\]
which also contradicts (\ref{8again}).  Therefore, $b_2^+(M^4)=1$.
\end{proof}

\begin{remark}
The preceding lemma implies that if we take $\epsilon\leq1$ in Theorem A, then $b_2^+(M^4)>0$ will show that $b_2^+(M^4)=1$.
Note that the work \cite{Don} and \cite{Fr} of Donaldson and Freedman will imply that the manifold is \emph{homeomorphic} to $\mathbb{CP}^2$ in this case.
\end{remark}

\begin{proof}[The Proof of Theorem C]  We can now combine Theorem \ref{ThmGur98} and Lemma \ref{Lemma21} to give the proof of Theorem C.  Assuming $b_2^{+}(M^4) \cdot b_2^{-}(M^4) > 0$, it follows from Lemma \ref{Lemma21} that $\beta(M^4,[g]) \geq 8$ for any metric $g \in \mathcal{Y}_2^{+}(M^4)$.  Moreover, if equality holds, then since $b_2^{+}(M^4) > 0$ we can argue as in the proof of Lemma \ref{Lemma21} to get equality in (\ref{rewritten}):
\begin{equation}\label{rewritten2}
  \int ||W_g^{+}||^2dv_g = \frac{4}{3}\pi^2(2\chi(M^4)+3\tau(M^4))+2\pi^2\tau(M^4).
\end{equation}
By Theorem \ref{ThmGur98} we conclude that $\tau(M^4) \geq 0$.  Reversing orientation and applying the same argument (since $b_2^{-}(M^4) > 0$) we also get $-\tau(M^4) \geq 0$, hence $\tau(M^4) = 0$.  Substituting this into (\ref{rewritten2}) implies that we have equality in (\ref{Gursky theorem}).  Therefore, $g$ is conformal to K\"ahler-Einstein metric $g_{KE}$.  By Proposition 2 of \cite{Der}, $\nabla W^{+}_{g_{KE}} \equiv 0$.

Applying the same argument with the opposite orientation, we see that $g$ is conformal to a K\"ahler-Einstein metric $g_{KE}'$.  By Obata's theorem \cite{Obata}, Einstein metrics are unique in their conformal class (except in the case of the sphere, which is ruled out in this case).  Therefore, $g_{KE} = g_{KE}'$, and since equality holds in (\ref{Gursky theorem}) with the opposite orientation it follows that $\nabla W^{-}_{g_{KE}} \equiv 0$.  We conclude that $g_{KE}$ is locally symmetric and Einstein; it follows from the classification of such spaces (for example, \cite{Jensen}) that $(M^4,g_{KE})$ is isometric to $(S^2 \times S^2, g_p)$, and Theorem C follows.
\end{proof}

\vskip.1in

\subsection{A preliminary lemma}  We end this section with a technical lemma that will be used in the proof of Theorem A.

\begin{lemma} \label{L2pinchLemma}
Let $(M^4,g)$ be a closed, compact oriented Riemannian four-manifold with $b_2^+(M^4)>0$ and
\begin{equation}\label{CP^2 pinching}
\beta(M^4,[g])=4(1+\epsilon)
\end{equation}
for some $0\leq\epsilon<1$.  Then
\begin{equation}\label{W^-}
\int_M||W^-||^2dv=\frac{6\epsilon}{2+\epsilon}\pi^2,
\end{equation}
\begin{equation}\label{W^+}
\int||W^+||^2dv=12\pi^2+\int||W^-||^2 \,dv,
\end{equation}
and
\begin{equation}\label{Yamabe constant}
Y(M^4,[g])\geq\frac{24\pi}{\sqrt{2+\epsilon}}.
\end{equation}
\end{lemma}

\begin{proof}
It follows from Lemma \ref{Lemma25} that $b_1(M^4)=0$, $b_2^+(M^4)=1,$ and $b_2^-(M^4)=0$. Therefore, $\chi(M^4)=3$ and $\tau(M^4)=1$.
By the Chern-Gauss-Bonnet and signature formulas, we have
\begin{equation}\label{GBC}
24\pi^2=\int||W||^2dv+4\int\sigma_2(P)dv,
\end{equation}
\begin{equation}\label{Sign}
12\pi^2=\int||W^+||^2dv-\int||W^-||^2dv.
\end{equation}
Since $\beta(M^4,[g])=4(1+\epsilon)$, we have
\begin{align}  \label{IS2}
4\int\sigma_2(P)dv = \frac{1}{1 + \epsilon}\int \| W \|^2 dv.
\end{align}
Substituting this into (\ref{GBC}) we conclude
\begin{align}  \label{allW}
\int \|W\|^2 dv = 24\left( \frac{1 + \epsilon}{2 + \epsilon}\right) \pi^2.
\end{align}
By (\ref{Sign}),
\begin{align*}
12\pi^2 &= \int||W^+||^2dv-\int||W^-||^2dv \\
&= \int \|W\|^2 dv - 2 \int \| W^{-}\|^2 dv \\
&= 24 \left( \frac{1 + \epsilon}{2 + \epsilon}\right) \pi^2 - 2 \int \| W^{-}\|^2 dv,
\end{align*}
which implies (\ref{W^-}).  Also, substituting (\ref{W^-}) into the signature formula (\ref{Sign}) we get (\ref{W^+}).

To prove (\ref{Yamabe constant}), we fist observe that (\ref{IS2}) and (\ref{allW}) imply
\begin{align*}
\int\sigma_2(P)dv = \left( \frac{6}{2+\epsilon}\right) \pi^2.
\end{align*}
Therefore, by (\ref{S2Y}),
\begin{align} \label{S2Y}
\left( \frac{6}{2+\epsilon}\right) \pi^2 = \int \sigma_2(P_g) \,dv_g \leq \frac{1}{96}Y(M^4,[g])^2,
\end{align}
and (\ref{Yamabe constant}) follows.
\end{proof}

\section{Modified Yamabe metrics}

As mentioned in the Introduction, the proof of Theorems A and B will use the Ricci flow.  We will use the fact that our assumptions are conformally invariant
and choose an initial metric that satisfies certain key estimates.   The metric will be a solution of a modified version of the Yamabe problem introduced in \cite{Gur00},
which we now review.

Let $(M^4,g)$ be a Riemannian four-manifold. Define
\begin{equation}\label{variant of scalar curvature}
F_g^+=R_g-2\sqrt{6}||W_g^+||,
\end{equation}
and
\begin{equation}\label{variant of conformal laplacian}
\mathcal{L}_g=-6\Delta_g+R_g-2\sqrt{6}||W^+||.
\end{equation}
$\mathcal{L}_g$ is a variant of conformal Laplacian that satisfies the following conformal transformation law:
\begin{equation}\label{conformal tranformation law}
\mathcal{L}_{\widetilde{g}}\phi=u^{-3}\mathcal{L}_g(\phi{u}),
\end{equation}
where $\widetilde{g}=u^2g\in[g]$. In analogy to the Yamabe problem, we define the functional
\begin{equation}\label{variant of Yamabe quotient}
\widehat{Y}_g[u]=\left\langle{u,\mathcal{L}_gu}\right\rangle_{L^2}/||u||_{L^4}^2,
\end{equation}
and the associated conformal invariant
\begin{equation}\label{variant of Yamabe constant}
\widehat{Y}(M^4,[g])=\inf_{u\in{W^{1,2}(M,g)}}\widehat{Y}_g[u].
\end{equation}
By the conformal transformation law of $\mathcal{L}_g$, the functional $u\to\widehat{Y}_g[u]$ is equivalent to the Riemannian functional
\begin{equation}\label{variant of definition}
\widetilde{g}=u^2g\to{vol(\widetilde{g})}^{-\frac{1}{2}}\int F_{\widetilde{g}}^+ \,dv_{\widetilde{g}}.
\end{equation}
The motivation for introducing this invariant is explained in the following result (see Theorem 3.3 and Proposition 3.5 of \cite{Gur00}):

\begin{theorem}
(i) Suppose $M^4$ admits a metric $g$ with $F_g^+\geq0$ on $M^4$ and $F_g^+>0$ somewhere. Then $b_2^+(M^4)=0$.
(ii) If $b_2^+(M^4)>0$, then $M^4$ admits a metric $g$ with $F_g^+\equiv0$ if and only if $(M^4,g)$ is a K\"ahler manifold with non-negative scalar curvature.
(iii) If $Y(M^4,[g])>0$ and $b_2^+(M^4)>0$, then $\widehat{Y}(M^4,[g])\leq0$ and there is a metric $\widetilde{g}=u^2g$ such that
\begin{equation}
F_{\widetilde{g}}^+=R_{\widetilde{g}}-2\sqrt{6}||W^+||_{\widetilde{g}}\equiv\hat{Y}(M,[g])\leq0
\end{equation}
and
\begin{equation}\label{Gap W+}
  \int R_{\widetilde{g}}^2 \,dv_{\widetilde{g}}\leq 24\int ||W^+_{\widetilde{g}}||^2 \,dv_{\widetilde{g}}.
\end{equation}
Furthermore, equality is achieved if and only if $F_{\widetilde{g}}^+\equiv0$ and $R_{\widetilde{g}}=2\sqrt{6}||W^+_{\widetilde{g}}||\equiv{const}$.
\end{theorem}

\begin{remark}
Recall that $F_g^+\equiv0$ on a K\"ahler manifold $(M^4,g)$ with $R\geq0$.
\end{remark}

\begin{remark}
In the rest of the paper, we will refer to the metric $\widetilde{g}$ in (iii) of Theorem 3.1, normalized to have unit volume, a {\em modified Yamabe metric}, and denote it by $g_m$. To simplify the notation, we write
$\hat{Y}(M,[g])=-\mu_+$. Then $g_m$ satisfies
\begin{equation}
R_{g_m}-2\sqrt{6}||W^+||_{g_m}=-\mu_+\leq0
\end{equation}
and (\ref{Gap W+}).
\end{remark}

As a preparation for the proof of Theorem A in next section, in the rest of this section we will list some preliminary curvature estimates of the modified Yamabe metric $g_m\in[g]$ with the assumption $b_2^+(M^4)>0$ and $\beta(M,[g])=4(1+\epsilon)$.

\begin{lemma}\label{initial metric}
Let $(M^4,g)$ be a closed, compact oriented Riemannian four-manifold with $b_2^+(M^4)>0$ and
\begin{equation}\label{CP^2 pinching}
\beta(M^4,[g])=4(1+\epsilon)
\end{equation}
for some $0<\epsilon<1$,
then we have for the modified Yamabe metric $g_m\in[g]$
\begin{equation}\label{g_m W-}
\int_M||W^-_{g_m}||^2dv_{g_m}=\frac{6\epsilon}{2+\epsilon}\pi^2,
\end{equation}
\begin{equation}\label{g_m W+}
\int||W^+_{g_m}||^2dv_{g_m}=12\pi^2+\int||W^-_{g_m}||^2dv_{g_m},
\end{equation}
\begin{equation}\label{g_m Yamabe}
Y(M,[g_m])\geq\frac{24\pi}{\sqrt{2+\epsilon}}
\end{equation}
\begin{equation}\label{g_m E}
\int|E_{g_m}|^2dv_{g_m}\leq6\int||W^-_{g_m}||dv_{g_m},
\end{equation}
\begin{equation}\label{g_m mu}
\frac{1}{12}\mu_+{Y}\leq3\int||W^-_{g_m}||^2dv_{g_m},
\end{equation}
\begin{equation}\label{g_m R average}
\frac{1}{24}\int(R_{g_m}-\bar{R}_{g_m})^2dv_{g_m}\leq3\int||W^-_{g_m}||^2dv_{g_m},
\end{equation}
where $\bar{R}_{g_m}=\int{R_{g_m}}dv_{g_m}$.
\end{lemma}
\begin{proof}
(\ref{g_m W-})(\ref{g_m W+})(\ref{g_m Yamabe}) follow from (\ref{W^-})(\ref{W^+})(\ref{Yamabe constant}) of Lemma 2.7 and confomal invariance.   The estimates (\ref{GBC}) and (\ref{Sign}) imply
\begin{equation}\label{difference}
\int||W^+||^2dv-4\int\sigma_2(P_{g_m})dv_{g_m}=3\int||W^-||_{g_m}^2dv_{g_m}.
\end{equation}
Recall the modified Yamabe metric stasfies $R_{g_m}+\mu_+=2\sqrt{6}||W^+_{g_m}||$. Squaring both sides of this formula and integrating over $M^4$, we have
\[\frac{1}{24}\int\left(R_{g_m}^2+2\mu_+R_{g_m}+\mu_+^2\right)dv_{g_m}=\int||W^+_{g_m}||^2dv_{g_m}.\]
With (\ref{difference}), we can rewrite this equation in the following way:
\[\frac{1}{2}\int|E_{g_m}|^2dv+\frac{1}{12}\mu_+\int{R_{g_m}}dv_{g_m}+\frac{1}{24}\int\mu_+^2dv_{g_m}=3\int||W^-_{g_m}||^2dv_{g_m}.\]
Since $\int{R_{g_m}}dv_{g_m}\geq Y(M^4,[g_m])>0$, (\ref{g_m E}) and (\ref{g_m mu}) follow from this equation.
To see (\ref{g_m R average}), we have
\begin{align*}
\frac{1}{24}\int(R_{g_m}-\bar{R}_{g_m})^2dv_{g_m} & =\frac{1}{24}\left(\int{R}^2_{g_m}dv_{g_m}-\bar{R}^2_{g_m}\right)\leq\frac{1}{24}\left(\int{R}^2_{g_m}dv_{g_m}-Y^2\right) \\
                                                                                         & \leq{\int||W^+_{g_m}||^2dv_{g_m}-4\int\sigma_2(P_{g_m})dv_{g_m}}=3\int||W^-_{g_m}||^2dv_{g_m}.
\end{align*}
\end{proof}

We end this section with a conformally invariant characterization of the Fubini-Study metric:

\begin{lemma}\label{CP2}
Let $(M^4,g)$ be a closed, compact oriented Riemannian four-manifold whose metric $g$ is of positive Yamabe type.
In addition, assume $b_1(M^4)=b_2^-(M^4)=0$ and $b_2^+(M^4)=1$. Then
\[\int_M\sigma_2(P) \,dv \leq12\pi^2\]
and equality holds if and only if $(M^4,g)$ is conformally equivalent to $(\mathbb{CP}^2,g_{FS})$.
\end{lemma}
\begin{proof}
By our assumptions we have $\chi(M^4)=3$ and $\tau(M^4)=1$.  Then the Gauss-Bonnet-Chern and signature formulas read
\begin{equation}
  24\pi^2=\int_M||W_g||^2 \,dv_g+4\int_M\sigma_2(P_g) \,dv_g,
\end{equation}
and
\begin{equation}
  12\pi^2=\int_M||W^+_g||^2\,dv_g-\int_M||W^-_g||^2\,dv_g.
\end{equation}
Therefore,
\begin{align*}
\int_M\sigma_2(P) dv &= 24 pi^2 - \int \| W \|^2 dv \\
&\leq 24\pi^2-\int_M||W^+||^2dv  \\
&=12\pi^2-\int_M||W^-||^2dv\leq12\pi^2.
\end{align*}
If equality holds then $W^-\equiv0$ and by (\ref{g_m E}) of Lemma \ref{initial metric} it immediately follows that $E_{g_m}=0$. Hence, $(M^4,g_m)$ is self-dual Einstein with positive scalar curvature and $b_2^+(M^4)=1$.
It is easy to check that the equality in (\ref{Gursky theorem}) is achieved, and therefore $g_m$ is conformal to a K\"ahler-Einstein metric.
By Obata's theorem, $(M^4,g_m)$ must be (K\"ahler-)Einstein with positive scalar curvature and $b_2^+(M^4)=1$.

Now $(M^4,g_m)$ is a complex surface with a positive K\"ahler-Einstein metric. For complex surfaces, (see Page 81 of \cite{B87})
\begin{equation}\label{relation of topological invariants}
3c_2(M^4)-c_1(M^4)^2=\chi(M^4)-3\tau(M^4)=0.
\end{equation}
Then the uniformlization of K\"ahler-Einstein manifolds (e.g. Theorem 2.13 in \cite{T}) implies that the universal cover of $(M,g_m)$ is (up to scaling) isometric to $(\mathbb{CP}^2,g_{FS})$.
Hence, $(M,g_m)$ is conformally equivalent to $(\mathbb{CP}^2,g_{FS})$ since $\mathbb{CP}^2$ does not have nontrivial smooth quotient space.
\end{proof}

\section{The Proofs of Theorems A and B}

Suppose $M^4$ is an oriented four-manifold with $b_2^{+}(M^4) > 0$ and $g$ is a metric on $M^4$ with $\beta(M^4,[g]) = 4(1 + \epsilon)$ for $\epsilon > 0$ small.  We want to show that if $\epsilon > 0$ is small
enough, then $M^4$ is diffeomorphic to $\mathbb{CP}^2$.  By Lemma \ref{initial metric}, the modified Yamabe metric $g_m \in [g]$ is close to a self-dual Einstein metric in an $L^2$-sense.  Using the Ricci flow, we
  want to `smooth' $g_m$ to obtain smallness of $W^{-}$ and $E$ in a {\em pointwise} sense.  We do this in two stages:  first, we show that for a small but uniform time, the Ricci flow applied to $g_m$ gives a metric
  for which $W^{-}$ and $E$ are small in an $L^p$-sense, for some $p > 2$.   Next, we appeal to a parabolic Moser iteration estimate of D. Yang \cite{Y} to conclude $L^\infty$-smallness.  The Bernstein-Bando-Shi estimates for the Ricci flow then imply bounds for $C^\infty$-norms of the curvature.   To complete the proof we apply a contradiction argument using a compactness result of Cheeger-Gromov-Taylor \cite{CGT}.

We begin with some definitions: on $(M^4,g)$, define

\begin{equation} \label{Gk}
G_k(g)=|E_g|^k+|R_g-\bar{R}_g|^k+||W^-_g||^k+|(F_g^+)_-|^k,
\end{equation}
where $(F_g^+)_-=\min(F_g^+,0)$. We shall suppress the subscript $g$ when there is no confusion.

\begin{lemma}\label{4.1}
Under the conditions of Theorem A, the modified Yamabe metric $g_m\in[g]$ satisfies $\int{G_2(g_m)dv_{g_m}}<c(\epsilon)$, where $c(\epsilon)\to0$ as $\epsilon\to0$.
\end{lemma}
\begin{proof}
This is a direct consequence of Lemma \ref{initial metric}.
\end{proof}

  Now recall some basic facts about the Ricci flow:
\begin{equation}\label{Ricci flow}
\left\{
\begin{array}{ll}
\frac{\partial{}}{\partial{t}}g=-2Ric(g) \\
g(0)=g_0
\end{array}
\right.
\end{equation}
The following short time-time existence result of Ricci flow has been established in  \cite{Ham82}.

\begin{proposition}
For arbitrary smooth metric $g_0$, there exists $T=T(g_0)$ such that (\ref{Ricci flow}) has a unique smooth solution for $t\in[0,T)$.
\end{proposition}

\begin{remark}
In general, the time interval $[0,T)$ depends on the initial metric $g_0$.
\end{remark}

Along the Ricci flow, define $G_k(t)=G_k(g(t))$.
The following estimate is of fundamental importance for our argument.

\begin{proposition}\label{evolution of G2}
Suppose we have a solution of the Ricci flow whose initial metric satisfies
\begin{equation}\label{initial smallness}
\int_MG_2(0)dv_{g(0)}=\frac{1}{2}\epsilon_0
\end{equation}
for some $\epsilon_0$ is sufficiently small.  Let
\[T=\inf\{{t}:\int_MG_2(t)dv_{g(t)}=\epsilon_0\}.\]
Assume in addition for $0\leq{t}\leq{T}$
\begin{equation}\label{Yamabe}
  Y(t) = Y(M^4,[g(t)]) \geq{b}>0,
\end{equation}
and
\begin{equation}\label{average scalar}
  0<\bar{R}(t)\leq{a}.
\end{equation}
Then for $0\leq{t}\leq{T}$, we have
\begin{equation}\label{evolution G_2}
  \frac{d}{dt}\int_MG_2(t)dv_{g(t)}\leq{\widetilde{a}}\int_MG_2(t)dv_{g(t)}-\widetilde{b}\left(\int_M{G_4(t)}dv_{g(t)}\right)^{1/2},
\end{equation}
where $\widetilde{a}$ and $\widetilde{b}$ are uniform positive constants independent of $\epsilon_0$.  Moreover, there exists $T_0$, which is independent of $\epsilon_0$ such that $T\geq{T_0}$, and we may choose $\widetilde{a}=\frac{4}{3}a$ and $\widetilde{b}=\frac{1}{12}b$.
\end{proposition}

\begin{proof}  


The proof is based on the evolution of the curvature under the Ricci flow in four dimensions, along with several algebraic inequalities.   We begin
by summarizing the evolution formulas we will need, most of which can be found in \cite{CK}:

\begin{lemma}
Under (\ref{Ricci flow}) on Riemannnian four-manifolds,
\begin{equation}\label{volume element}
\frac{\partial}{\partial{t}}dv=-R{dv},
\end{equation}
\begin{equation}\label{E^2}
  \frac{\partial}{\partial{t}}|E|^2=\Delta|E|^2-2|\nabla{E}|^2+4WEE-4trE^3+\frac{2}{3}R|E|^2,
\end{equation}
\begin{equation}\label{R^2}
  \frac{\partial}{\partial{t}}(R^2)=\Delta{(R^2)}-2|\nabla{R}|^2+4R|E|^2+R^3,
\end{equation}
\begin{equation}\label{W+-^2}
\frac{\partial}{\partial{t}}||W^{\pm}||^2={\Delta||W^{\pm}||^2-2||\nabla{W^{\pm}}||^2+36\det{W^{\pm}}+W^{\pm}EE,}
\end{equation}
\begin{equation}\label{W+- inequality}
\frac{\partial}{\partial{t}}||W^{\pm}||\leq{\Delta||W^{\pm}||+\sqrt6||W^{\pm}||^2+\frac{\sqrt{6}}{6}|E|^2,}
\end{equation}
\begin{equation}\label{F inequality}
\frac{\partial}{\partial{t}}((F^+)_-)\geq\Delta{((F^+)_-)}-((F^+)_-)^2+2R(F^+)_-
\end{equation}
where
\[W^{\pm}EE:=W^{\pm}_{ijkl}E_{ik}E_{jl},\,\,\, F^+_g:=R_g-2\sqrt{6}||W^+||,\,\,\,(F^+)_{-}:=\min\{F_g^+,0\}.\]
\end{lemma}

\begin{remark}  For the evolution formulas of $W^{\pm}$ we rely on unpublished notes of D. Knopf \cite{Knopf}.  \end{remark}

As a corollary of the formulas above we have in four dimensions:

\begin{corollary}\label{integral}
Under (\ref{Ricci flow}),
\begin{equation}\label{volume}
\frac{d}{d{t}}\int{dv}=-\int{R}{dv},
\end{equation}
\begin{equation}\label{L2 E}
\frac{d}{dt}\int|E|^2dv=\int\left(-2|\nabla{E}|^2+4WEE-4trE^3-\frac{1}{3}R|E|^2\right)dv,
\end{equation}
\begin{equation}\label{L2 R-barR}
\frac{d}{dt}\int(R-\bar{R})^2dv=\int\left(-2|\nabla{R}|^2+4(R-\bar{R})|E|^2+\bar{R}(R-\bar{R})^2\right)dv,
\end{equation}
\begin{multline}\label{L2 F-}
\frac{d}{dt}\int|(F^+)_{-}|^2dv\leq-2\int\left(|\nabla{(F^+)_{-}}|^2+\frac{R}{6}|(F^+)_{-}|^2\right)dv\\
  -\int((F^+)_{-})^3dv+\int\left(\frac{4}{3}\bar{R}|(F^+)_{-}|^2+\frac{4}{3}(R-\bar{R})|(F^+)_{-}|^2\right)dv,
\end{multline}
\begin{multline}\label{L2 W-}
\frac{d}{dt}\int||W^{-}||^2dv\leq-\int\left(2|\nabla||{W^-}|||^2+R||W^-||^2\right)dv\\
+\int\left(2\sqrt{6}||W^-||^3+\frac{\sqrt6}{3}||W^-|||E|^2\right)dv,
\end{multline}
where
\[\bar{R}=\int{R}dv\Big/\int{dv},\,\,\,F^+_g=R_g-2\sqrt{6}||W^+||,\,\,\,(F^+)_{-}=\min\{F_g^+,0\}.\]
\end{corollary}

For the proof of Proposition \ref{evolution of G2} we will also need some algebraic inequalities.  The first appears in (\cite{CGY03}, Lemma 4.3), and is based on (\cite{Mar98}, Lemma 6):

\begin{lemma} \label{bew}
\begin{equation}\label{WEE ||W|| and |E|}
  WEE\leq\frac{\sqrt{6}}{3}\left(||W^+||+||W^-||\right)|E|^2
\end{equation}
\end{lemma}

\begin{proof}  Recall the well-known decomposition of Singer-Thorpe:
\begin{equation}
Riem=
\begin{pmatrix}
W^++\frac{R}{12}Id & B \\
                       B^{*}   &  W^-+\frac{R}{12}Id
\end{pmatrix}
\end{equation}
Note the compositions satisfy
\[BB^*:\Lambda^2_+\to\Lambda_+^2,\quad B^*B:\Lambda_-^2\to\Lambda_-^2.\]

Fix a point $P\in{M^4}$, and let $\lambda_1^{\pm}\leq\lambda_2^{\pm}\leq\lambda_{3}^{\pm}$ denote the eigenvalues of $W^{\pm}$, where ${W^{\pm}}$ are interpreted as endomorphisms of $\Lambda^2_{\pm}$. Also denote the eigenvalues of $BB^*:\Lambda^2_+\to\Lambda_+^2$ by
$b_1^2\leq{b_2^2}\leq{b_3^2}$, where $0\leq{b_1}\leq{b_2}\leq{b_3}$. From Lemma 4.3  of \cite{CGY03}, we have

\begin{equation}\label{WEE}
  WEE\leq4\left(\sum_{i=1}^3\lambda_i^+b_i^2+\sum_{i=1}^3\lambda_i^-b_i^2\right)
\end{equation}
Recall from Lemma 4.2 of \cite{CGY03} that $|E|^2=4\sum_{i=1}^3b_i^2$. For a trace-free $3\times3$ matrix $A$, we have the sharp inequality:
\begin{equation}\label{sharp33}
|A(X,X)|\leq{\frac{\sqrt{6}}{3}||A|||X|^2}.
\end{equation}
Apply (\ref{sharp33}) to $A=diag(\lambda_1^{\pm},\lambda_2^{\pm},\lambda_3^{\pm})$ and $X=(b_1,b_2,b_3)$. We derive
\begin{equation}\label{sharp}
4\left(\sum_{i=1}^3\lambda_i^+b_i^2+\sum_{i=1}^3\lambda_i^-b_i^2\right)\leq\frac{\sqrt{6}}{3}\left(||W^+||+||W^-||\right)|E|^2.
\end{equation}
Combining (\ref{WEE}) and (\ref{sharp}), we derive the desired inequality.
\end{proof}

Now we turn to the proof of Proposition \ref{evolution of G2}. From the definition of $G_k$, we have
\begin{equation}
\frac{d}{dt}\int{G_2}dv= \frac{d}{dt}\int|E|^{2}dv+\frac{d}{dt}\int|R-\bar{R}|^{2}dv+\frac{d}{dt}\int||W^-||^{2}dv+\frac{d}{dt}\int|(F^+)_-|^{2}dv.
\end{equation}

Now estimate each term of the right hand side from formulas in Corollary \ref{integral}.

\begin{equation} \label{E2E4}
\begin{split}
\frac{d}{dt}\int|E|^2dv &  \leq\int\left(-2|\nabla{E}|^2-\frac{1}{3}R|E|^2-4trE^3+\frac{4\sqrt{6}}{3}(||W^+||+||W^-||)|E|^2\right)dv\\
                                       &  \leq-\frac{Y}{3}\left(\int|E|^4dv\right)^{1/2}+4\left(\int|E|^2dv\right)^{1/2}\left(\int|E|^4dv\right)^{1/2} \\
                                       &  +\frac{4\sqrt{6}}{3}\left(\int||W^-||^2dv\right)^{1/2}\left(\int|E|^4dv\right)^{1/2}\\
                                       & +\frac{2}{3}\left(\int|(F^+)_-|^2dv\right)^{1/2}\left(\int|E|^4dv\right)^{1/2}+\frac{2}{3}\int(R-\bar{R})|E|^2dv \\
                                       & +\frac{2}{3}\int\bar{R}|E|^2dv\\
                                       &  \leq{-\frac{Y}{6}}\left(\int|E|^4dv\right)^{1/2}+\frac{2}{3}\bar{R}\int|E|^2dv.
\end{split}
\end{equation}

The first inequality follows from Lemma \ref{bew}.  The second inequality follows from Cauchy-Schwartz inequality and the conformally invariant Sobolev inequality:
\begin{align*}
Y \Big( \int \phi^4 dv \Big)^{1/2} \leq \int \big( |\nabla \phi|^2 + \frac{1}{6}R \phi^2 \big) dv.
\end{align*}
The third inequality follows from the smallness assumption of $\epsilon_0$.  Next, we esimate
\begin{equation}\label{R2R4}
\begin{split}
\frac{d}{dt}\int(R-\bar{R})^2dv  & = \int\left(-2|\nabla(R-\bar{R})|^2-\frac{1}{3}R(R-\bar{R})^2\right)dv+\frac{1}{3}\int(R-\bar{R})^3dv\\
                                                      & +\frac{4}{3}\int\bar{R}(R-\bar{R})^2dv+4\int(R-\bar{R})|E|^2dv \\
                                                      & \leq{-\frac{Y}{3}\left(\int(R-\bar{R})^4dv\right)^{1/2}}+\frac{1}{3}\left(\int(R-\bar{R})^2dv\right)^{1/2}\left(\int(R-\bar{R})^4dv\right)^{1/2}\\
                                                      & +\frac{4}{3}\bar{R}\int(R-\bar{R})^2dv+4\left(\int(R-\bar{R})^2dv\right)^{1/2}\left(\int|E|^4dv\right)^{1/2} \\
                                                      & \leq{-\frac{Y}{6}\left(\int(R-\bar{R})^4dv\right)^{1/2}}+\frac{4}{3}\bar{R}\int(R-\bar{R})^2dv\\
                                                      & +4\left(\int(R-\bar{R})^2dv\right)^{1/2}\left(\int|E|^4dv\right)^{1/2}.
\end{split}
\end{equation}

The first inequality is from Sobolev inequality and Cauchy-Schwartz inequality and the second inequality is from the smallness assumption of $\epsilon_0$.

\begin{equation}\label{F2F4}
\begin{split}
\frac{d}{dt}\int|(F^+)_-|^2dv &  \leq-\frac{Y}{3}\left(\int|(F^+)_-|^4dv\right)^{1/2}+\left(\int|(F^+)_-|^2dv\right)^{1/2}\left(\int|(F^+)_-|^4dv\right)^{1/2} \\
                                                   & +\frac{4}{3}\bar{R}\int|(F^+)_-|^2dv +\frac{4}{3}\left(\int(R-\bar{R})^2dv\right)^{1/2}\left(\int|(F^+)_-|^4dv\right)^{1/2} \\
                                                   &  \leq-\frac{Y}{6}\left(\int|(F^+)_-|^4dv\right)^{1/2}+\frac{4}{3}\bar{R}\int|(F^+)_-|^2dv
\end{split}
\end{equation}

The first inequality is from Sobolev inequality and Cauchy-Schwartz inequality and the second inequality is from the smallness assumption of $\epsilon_0$.

\begin{equation}\label{W2W4}
\begin{split}
\frac{d}{dt}\int||W^-||^2dv &   \leq -\frac{Y}{3}\left(||W^-||^4\right)^{1/2}+\frac{2}{3}\left(\int(R-\bar{R})^2dv\right)^{1/2}\left(\int||W^-||^4dv\right)^{1/2}\\
                                               &   +2\sqrt{6}\left(\int||W^-||^2dv\right)^{1/2}\left(\int||W^-||^4dv\right)^{1/2}\\
                                               &   +\frac{\sqrt{6}}{3}\left(\int||W^-||^2dv\right)^{1/2}\left(\int|E|^4dv\right)^{1/2} \\
                                               &   \leq -\frac{Y}{6}\left(||W^-||^4\right)^{1/2}+\frac{\sqrt{6}}{3}\left(\int||W^-||^2dv\right)^{1/2}\left(\int|E|^4dv\right)^{1/2}
\end{split}
\end{equation}

The first inequality is from Sobolev inequality and Cauchy-Schwartz inequality and the second inequality is from the smallness assumption of $\epsilon_0$.

With (\ref{E2E4})(\ref{R2R4})(\ref{F2F4})(\ref{W2W4}), it is now easy to see from the smallness assumption of $\epsilon_0$

\begin{equation}
\frac{d}{dt}\int_MG_2dv\leq\frac{4}{3}a\int_MG_2dv-\frac{1}{12}b\left(\int{G_4}dv\right)^{1/2}.
\end{equation}

Take $\widetilde{a}=\frac{4}{3}a$ and $\widetilde{b}=\frac{1}{12}b$. Clearly, we have proved the desired inequality. Note that the differential inequality
\[\frac{d}{dt}\int_MG_2dv\leq\frac{4}{3}a\int_MG_2dv\]
implies
\[T\geq{T}_0=\frac{3\log2}{4}a.\]
\end{proof}

\begin{remark}
The importance of this lemma is that $T_0$ does \emph{not} depend on $\epsilon_0$, which implies that we may evolve the Ricci flow on a uniform time interval once we derive uniform bounds for curvatures.
\end{remark}

With Lemma \ref{4.1}, it is easy to see that we may choose $\epsilon$ in Theorem A sufficiently small so that (\ref{initial smallness}) is satisfied. To apply Proposition \ref{evolution of G2} and Proposition \ref{evolution of G3}, we also need establish (\ref{Yamabe}) and (\ref{average scalar}).
We now establish these inequalities and prove the following proposition.

\begin{proposition}\label{choice of initial metric}
Suppose the initial metric of Ricci flow is chosen as the modified Yamabe metric for sufficiently small $\epsilon$ in Theorem A. Then there exists $\widetilde{T}$ which does not depend on $\epsilon$ such that for $0\leq{t}\leq{\widetilde{T}}$ all conditions of Proposition \ref{evolution of G2} are satisfied.
\end{proposition}
\begin{proof}

It is clear from Lemma \ref{4.1} that if we choose sufficiently small $\epsilon$ in Theorem A, we can establish (\ref{initial smallness}) for arbitrary small $\epsilon_0$. On $0\leq{t}\leq{T}$, (\ref{initial smallness}) implies that
\begin{equation}\label{smallness}
\int|E|^2dv\leq\frac{1}{2}\epsilon_0,\,\,\, \int||W^-||^2dv\leq\frac{1}{2}\epsilon_0,\,\,\,\int(R-\bar{R})^2dv\leq\frac{1}{2}\epsilon_0.
\end{equation}
Recall $b_1(M)=0$, $b_2^-(M)=0$, and $b_2^+(M)=1$. From signature and Chern-Gauss-Bonnet formula, we obtain
\begin{equation}
\int||W^+||^2dv=12\pi^2+\int||W^-||^2dv\leq12\pi^2+\frac{1}{2}\epsilon_0
\end{equation}
and thereby
\begin{equation}\label{sigma-2}
12\pi^2\geq4\int\sigma_2(P)dv=24\pi^2-\int||W||^2dv\geq{12\pi^2-\epsilon_0}.
\end{equation}
Now with the same argument in Lemma \ref{L2pinchLemma}, we can derive
\begin{equation}
\frac{1}{96}Y(t)^2\geq\int\sigma_2(P_{g(t)})dv_{g(t)}\geq2\pi^2
\end{equation}
if we choose sufficiently small $\epsilon_0$. Since the initial metric is of positive Yamabe type and the square of Yamabe constant has a strictly positive lower bound, we have established (\ref{Yamabe}).

Note that (\ref{smallness}) and (\ref{sigma-2}) imply that for some $C>0$
\begin{equation}
\frac{1}{C^2}\leq\int\bar{R}^2_{g(t)}dv_{g(t)}=\bar{R}^2_{g(t)}vol(M,g(t))\leq{C^2},\,\,\, \frac{1}{C^2}\leq\int{R}^2_{g(t)}dv_{g(t)}\leq{C^2}..
\end{equation}
To establish (\ref{average scalar}), it now suffices to derive a uniform lower bound for the volume since $\int{\bar{R}^2}dv$.
\begin{lemma}\label{volume}
Under conditions of Proposition \ref{choice of initial metric} with $vol(M,g(0))=1$, along Ricci flow, there exists constant $T^{'}>0$ such that for $0\leq{t}\leq{T^{'}}\leq{T}$
\begin{equation}\label{volume estimate}
\frac{9}{4}\geq{vol}(M,g(t))\geq\frac{1}{4}.
\end{equation}
In addition, there exists $T_1>0$  which does not depend on $\epsilon_0$ such that $T^{'}>T_1$.
\end{lemma}
\begin{proof}
Recall the evolution equation for volume under the Ricci flow:
\begin{equation}
\frac{d}{d{t}}\int{dv}=-\int{R}{dv}.
\end{equation}
Hence, for $0\leq{t}\leq{T}$
\[\frac{d}{d{t}}\int{dv}\geq-\int{|R|}{dv}\geq-{\left(\int{R^2}dv\right)^{1/2}\left(\int{dv}\right)^{1/2}\geq-{C}\left(\int{dv}\right)^{1/2}}.\]
Similarly,
\[\frac{d}{d{t}}\int{dv}\leq{C}\left(\int{dv}\right)^{1/2}.\]
It is then easy to derive
\[|\sqrt{vol(t)}-\sqrt{vol(0)}|\leq{Ct}.\]
From this inequality, it is easy to choose $T^{'}=\min\{\frac{1}{2C},T\}$ such that (\ref{volume estimate}) is satisfied. It is easy to see such a $T_1$ exists since $T\geq{T_0}$, where $T_0$ does not depend on $\epsilon_0$.
\end{proof}
With Lemma \ref{volume}, we establish (\ref{average scalar}) on $0\leq{t}\leq{T'}\leq{T}$. If we choose $\widetilde{T}=T'$, all conditions of Proposition \ref{evolution of G2} are satisfied on $0\leq{t}\leq{\widetilde{T}}$ and clearly $\widetilde{T}$ has a positive universal lower bound.

\end{proof}

We now derive integral estimates for $G_3$. For the sake of clearness, we first establish the estiamtes for $\int|E|^3dv$ and then derive a similar evolution inequality for $\int{G_3}dv$ as we did in Lemma \ref{evolution of G2}.

\begin{lemma}\label{E3}
Under the conditions of Proposition \ref{evolution of G2}, along the Ricci flow, we have
\[\int|E_{g(t)}|^3dv_{g(t)}\leq{{C\epsilon_0^{3/2}}{t^{-1}}},\]
for any $0<t\leq{\widetilde{T}}$, where $C$ is a universal constant which does not depend on $\epsilon_0$.
\end{lemma}
\begin{proof}
It is clear from (\ref{E2E4}) that
\begin{equation}\label{|E|^2 evolution}
\frac{d}{dt}\int|E|^2dv+C\left(\int|E|^4dv\right)^{1/2}\leq{C\int|E|^2dv},
\end{equation}
where $C$ is a constant which does not depend on $\epsilon_0$.
From (\ref{E^2}), we can compute
\begin{equation}\label{E^p}
 \frac{\partial}{\partial{t}}|E|^p=\frac{p}{2}|E|^{p-2}\frac{\partial}{\partial{t}}|E|^2=\frac{p}{2}|E|^{p-2}\left(\Delta|E|^2-2|\nabla{E}|^2+4WEE-4trE^3+\frac{2}{3}R|E|^2\right),
\end{equation}
Note that
\[\Delta|E|^p=\frac{p}{2}|E|^{p-2}\Delta|E|^2+p(p-2)|E|^{p-2}|\nabla|E||^2.\]
From this identity and (\ref{E^p}), it is easy to derive for $p\geq3$
\begin{equation}\label{|E|^p integral}
\begin{split}
\frac{d}{dt}\int|E|^pdv  & =\int\left(\frac{\partial}{\partial{t}}|E|^p-R|E|^p\right)dv \\
                                        &  \leq\int\left(-p(p-2)|E|^{p-2}|\nabla{|E|}|^2+2p|E|^2WEE\right)dv \\
                                        &  +\int\left(-2p|E|^{p-2}trE^3+\left(\frac{p}{3}-1\right)R|E|^p\right)dv \\
                                        &  =\int\left(-\frac{4(p-2)}{p}|\nabla{|E|^{\frac{p}{2}}}|^2-\frac{2(p-2)}{3p}R|E|^p\right)dv \\
                                        &  +\int\left(2p|E|^2WEE-2p|E|^{p-2}trE^3+C_pR|E|^p\right)dv
\end{split}
\end{equation}
We now estimate the terms in last line of (\ref{|E|^p integral})
\begin{equation}
\int{||E|^{p-2}trE^3|}dv\leq\left(\int|E|^2dv\right)^{1/2}\left(\int|E|^{2p}dv\right)^{1/2}
\end{equation}
\begin{equation}
\begin{split}
\int|E|^{p-2}WEEdv & \leq\frac{\sqrt{6}}{3}\int\left(||W^-||+||W^+||\right)|E|^pdv \\
                                  & \leq\frac{\sqrt{6}}{3}\left(\int||W^-||^2dv\right)^{1/2}\left(\int|E|^{2p}dv\right)^{1/2} \\
                                  & +\frac{1}{6}\left(\int((F^+)_-)^2dv\right)^{1/2}\left(\int|E|^{2p}dv\right)^{1/2}+\frac{1}{6}\int{R|E|^p}dv
\end{split}
\end{equation}
\begin{equation}
\begin{split}
\int{|R||E|^p}dv & \leq\int|R-\bar{R}||E|^pdv+\bar{R}\int|E|^pdv\\
                            & \leq\left(\int(R-\bar{R})^2dv\right)^{1/2}\left(\int|E|^{2p}dv\right)^{1/2}+\bar{R}\int|E|^pdv
\end{split}
\end{equation}
From the smallness assumption of $\epsilon_0$, it is now easy to derive
\begin{equation}\label{|E|^p evolution}
\frac{d}{dt}\int|E|^pdv+C_p\left(\int|E|^{2p}dv\right)^{1/2}\leq{C_p\int|E|^pdv}.
\end{equation}
In this proof, we shall only need this formula for $p=2,3$.
Take two smooth cut-off functions $\phi_1$ and $\phi_2$ such that $0\leq\phi_i\leq1$ on $[0,\widetilde{T}]$ for $i=1,2$.
Take $\tau<\tau'<\widetilde{T}$.
For $\phi_1$, we choose $0\leq\phi_1\leq1$ on $[0,\tau]$ and $\phi_1\equiv1$ on $[\tau,\widetilde{T}]$.
For $\phi_2$, we choose $\phi_2\equiv0$ on $[0,\tau]$, $0\leq\phi_2\leq1$ on $[\tau,\tau']$, and $\phi_1\equiv1$ on $[\tau',\widetilde{T}]$.
Also assume $|\phi_i'|_{L^\infty}$ have appropriate bound.
It is now easy to derive
\begin{equation}\label{|E|^p evolution cutoff}
\frac{d}{dt}\left(\phi_i\int|E|^pdv\right)+C\phi_i\left(\int|E|^{2p}dv\right)^{1/2}\leq{C\left(\phi_i+|\phi_i'|\right)\int|E|^pdv},
\end{equation}
for $i=1,2$ and $p=2,3$.

Set $p=2$ and $i=1$. Integrate (\ref{|E|^p evolution cutoff}) over $[0,t]$ for some $t>\tau'$:

\begin{equation}\label{E^2 integration}
\int|E|^2dv+C\int_\tau^t\left(\int|E|^4dv\right)^{1/2}ds\leq{C\left(1+\frac{1}{\tau}\right)\int_0^t\left(\int|E|^2dv\right)ds}
\end{equation}

Set $p=3$ and $i=2$. Integrate (\ref{|E|^p evolution cutoff}) over $[0,t]$ for some $t>\tau'$:

\begin{equation}\label{E^3 integration}
\int|E|^3dv+C\int_{\tau'}^t\left(\int|E|^6dv\right)^{1/2}ds\leq{C\left(1+\frac{1}{\tau'-\tau}\right)\int_\tau^t\left(\int|E|^3dv\right)ds}
\end{equation}

Now we have
\begin{align*}
\int_\tau^t\left(\int|E|^3dv\right)ds &\leq \int_\tau^t\left(\int|E|^2dv\right)^{1/2}\left(\int|E|^4dv\right)^{1/2}ds \\
                                                    &\leq C\epsilon_0^{1/2}\int_\tau^t\left(\int|E|^4\right)^{1/2}ds \\
                                                    &\leq C\epsilon_0^{1/2}\left(1+\frac{1}{\tau}\right)\int_0^t\left(\int|E|^2dv\right)ds
\end{align*}
where second line follow from taking $\epsilon=\frac{1}{2}\epsilon_0$ in Lemma 5.5 and third line follows from (\ref{E^2  integration}).
It then follows
\begin{align*}
\int|E|^3dv  &\leq C\left(1+\frac{1}{\tau'-\tau}\right)\int_\tau^t\left(\int|E|^3dv\right)ds \\
                   &\leq C\epsilon_0^{1/2}\left(1+\frac{1}{\tau}\right)\left(1+\frac{1}{\tau'-\tau}\right)\int_0^t\left(\int|E|^2dv\right)ds \\
                   &\leq C\epsilon_0^{3/2}\left(1+\frac{1}{\tau}\right)\left(1+\frac{1}{\tau'-\tau}\right)t
\end{align*}
Take $\tau=\frac{1}{4}t$ and $\tau'=\frac{1}{2}t$ and we get desired estimate. In particular, if we choose $t\in[\frac{1}{4}\widetilde{T},\widetilde{T}]$, we have
\begin{equation}\label{E3 small}
\sup_{\widetilde{T}/4\leq{t}\leq{\widetilde{T}}}\int|E|^3dv\leq{C\epsilon_0^{3/2}}.
\end{equation}
\end{proof}

Now we prove an evolution inequality for $\int{G_3}dv$ similar as (\ref{evolution G_2}).

\begin{proposition}\label{evolution of G3}
Under the same conditions of Proposition \ref{evolution of G2}, for $\frac{1}{4}\widetilde{T}\leq{t}\leq{\widetilde{T}}$, we have
\begin{equation}\label{evolution G_3}
  \frac{d}{dt}\int_MG_3(t)dv_{g(t)}\leq{\widetilde{a}'}\int_MG_3(t)dv_{g(t)}-\widetilde{b}'\left(\int_M{G_6(t)}dv_{g(t)}\right)^{1/2},
\end{equation}
where $\widetilde{a}'$ and $\widetilde{b}'$ are uniform positive constants independent of $\epsilon_0$ and we may choose $\widetilde{a}'=c_1a$ and $\widetilde{b}'=c_2b$, where $c_1$ and $c_2$ are universal positive constants.
\end{proposition}

\begin{proof}
The proof is similar to that of (\ref{evolution G_2}).  From the definition of $G_k$, we have
\begin{equation}\label{G_3}
\frac{d}{dt}\int{G_3}dv= \frac{d}{dt}\int|E|^{3}dv+\frac{d}{dt}\int|R-\bar{R}|^{3}dv+\frac{d}{dt}\int||W^-||^{3}dv+\frac{d}{dt}\int|(F^+)_-|^{3}dv.
\end{equation}
Note that we shrink the time interval to $[\frac{1}{4}\widetilde{T},\widetilde{T}]$, so from the previous lemma, we have known that $\int|E|^3$ is bounded by $c(\epsilon_0)$,
where $c(\epsilon_0)\to0$ as $\epsilon_0\to0$.

We now estimate the right hand side of (\ref{G_3}) term by term

\begin{equation} \label{E3E6}
\begin{split}
\frac{d}{dt}\int|E|^3dv &  =\int\left(\frac{\partial}{\partial{t}}|E|^3-R|E|^3\right)dv \\
                                       &  = \int\left(\frac{3}{2}|E|\left(\Delta|E|^2-2|\nabla{E}|^2+4WEE-4trE^3\right)\right)dv \\
                                       &  \leq{-c_2Y}\left(\int|E|^6dv\right)^{1/2}+c_1{\bar{R}}\int|E|^3dv,
\end{split}
\end{equation}
where the inequality is established similarly to (\ref{E2E4}).

\begin{equation} \label{W3W6}
\begin{split}
\frac{d}{dt}\int||W^-||^3dv & =\int\left(\frac{\partial}{\partial{t}}||W^-||^3-R||W^-||^3\right)dv \\
                                               & \leq\int\left(3||W^-||^2\Delta||W^-||+3\sqrt{6}||W^-||^3+\frac{\sqrt{6}}{2}||W^-||^2|E|^2-R||W^-||^3\right)dv
\end{split}
\end{equation}

Recall convexity inequality:
\begin{equation}\label{convex}
ab\leq{\frac{a^p}{p}+\frac{b^q}{q}}
\end{equation}
for $a,b\geq0$ and $1/p+1/q=1$.
Take $p=3/2$, $q=3$, $a=\left(\int||W^-||^6dv\right)^{1/3}$ and $b=\left(\int|E|^3dv\right)^{2/3}$. We have
\begin{equation}\label{convex W}
\begin{split}
\int||W^-||^2|E|^2dv&  \leq\left(\int||W^-||^6dv\right)^{1/3}\left(\int|E|^3dv\right)^{2/3} \\
                                    &  \leq{\frac{2K}{3}\left(\int||W^-||^6dv\right)^{1/2}+\frac{1}{3K}\left(\int|E|^3dv\right)^2}
\end{split}
\end{equation}
We may take $K$ to be a small multiple of the Yamabe constant and absorb the first term of (\ref{convex W}) by the Soblev inequality.
The second term of (\ref{convex W}) is bounded by a constant multiple of $\int|E|^3dv$ from (\ref{E3 small}) and the smallness assumption of $\epsilon_0$.
It is then easy to derive from Sobolev inequality and Cauchy-Schwartz inequality
\begin{equation}
\frac{d}{dt}\int||W^-||^3dv \leq {-c_2Y}\left(\int||W^-||^6dv\right)^{1/2}+c_1{\bar{R}}\int||W^-||^3dv+C\int|E|^3dv
\end{equation}

\begin{equation}\label{F3F6}
\begin{split}
\frac{d}{dt}\int|(F^+)_-|^3dv &  =-\frac{d}{dt}\int((F^+)_-)^3dv\\
                                                   & =\int\left(-3((F^+)_-)^2\frac{\partial}{\partial{t}}((F^+)_-)-R|(F^+)_-|^3\right)dv\\
                                                   & \leq  {-c_2Y}\left(\int|(F^+)_-|^6dv\right)^{1/2}+c_1{\bar{R}}\int|(F^+)_-|^3dv
\end{split}
\end{equation}
The inequality follows from evolution inequality (\ref{F inequality}), Sobolev inequality and the following trick:
\begin{equation}
\begin{split}
\int{R}|(F^+)_-|^3dv & \leq{\int{(R-\bar{R})}|(F^+)_-|^3dv}+\bar{R}\int|(F^+)_-|^3dv \\
                                     & \leq{\left(\int(R-\bar{R})^2\right)^{1/2}\left(\int|(F^+)_-|^6dv\right)^{1/2}}+C\int|(F^+)_-|^3dv
\end{split}
\end{equation}
Note the first term of the last line can be absorbed by Sobolev inequality with the smallness asssumption of $\epsilon_0$.

\begin{equation}\label{R3R6}
\begin{split}
\frac{d}{dt}\int|R-\bar{R}|^3dv & = \int\left(\frac{\partial}{\partial{t}}|R-\bar{R}|^3-R|R-\bar{R}|^3\right)dv \\
                                                     & = \int\left(\frac{3}{2}|R-\bar{R}|\frac{\partial}{\partial{t}}(R-\bar{R})^2-R|R-\bar{R}|^3\right)dv
\end{split}
\end{equation}
\begin{equation}\label{(R-Rbar)^2}
\begin{split}
\frac{\partial}{\partial{t}}(R-\bar{R})^2 & = \frac{\partial}{\partial{t}}R^2 -2\frac{\partial}{\partial{t}}(R\bar{R})+\frac{\partial}{\partial{t}}\bar{R}^2\\
                                                                  & = \Delta{(R^2)}-2|\nabla{R}|^2+4R|E|^2+R^3-2R\frac{\partial}{\partial{t}}\bar{R}-2\bar{R}\frac{\partial}{\partial{t}}R+2\bar{R}\frac{\partial}{\partial{t}}\bar{R} \\
                                                                  & = \Delta{(R-\bar{R})^2}-2|\nabla{(R-\bar{R})}|^2+4(R-\bar{R})|E|^2+R^2(R-\bar{R})-2(R-\bar{R})\frac{\partial}{\partial{t}}\bar{R}
\end{split}
\end{equation}
Recall the evolution equation of $\bar{R}$ under Ricci flow:
\begin{equation}\label{Rbar}
\frac{\partial}{\partial{t}}\bar{R}=\frac{1}{vol}\int\left(2|E|^2-\frac{1}{2}R^2\right)dv+\bar{R}^2.
\end{equation}
Plugging (\ref{(R-Rbar)^2}) and (\ref{Rbar}) into (\ref{R3R6}), we can derive
\begin{equation}\label{R3R6 plugin}
\begin{split}
\frac{d}{dt}\int|R-\bar{R}|^3dv  & \leq\int\left(\frac{3}{2}|R-\bar{R}|\Delta(R-\bar{R})^2-R|R-\bar{R}|^3\right)dv \\
                                                      & +6\int(R-\bar{R})^2|E|^2dv+\frac{3}{2}\int|R^2-\bar{R}^2|(R-\bar{R})^2dv \\
                                                      & +\frac{3}{2}\left|\bar{R}^2-\frac{\int{R^2}dv}{vol}\right|\int(R-\bar{R})^2dv+\frac{6}{vol}\int|E|^2dv\int(R-\bar{R})^2dv.
\end{split}
\end{equation}
Now we estimate the terms in second and third line of (\ref{R3R6 plugin}):
\begin{equation}
\begin{split}\label{R2E2}
\int(R-\bar{R})^2|E|^2dv  &  \leq\left(\int|R-\bar{R}|^6dv\right)^{1/3}\left(\int|E|^3dv\right)^{2/3} \\
                                            &  \leq{\frac{2K}{3}\left(\int|R-\bar{R}|^6dv\right)^{1/2}+\frac{1}{3K}\left(\int|E|^3dv\right)^2}
\end{split}
\end{equation}
We may take $K$ to be a small multiple of the Yamabe constant and absorb the first term of (\ref{R2E2}) by the Soblev inequality.
The second term of (\ref{R2E2}) is bounded by a constant multiple of $\int|E|^3dv$ from (\ref{E3 small}) and the smallness assumption of $\epsilon_0$.
\begin{equation}\label{R2E2 integral}
\begin{split}
\int|E|^2dv\int(R-\bar{R})^2dv & \leq{C}{\left(\int|E|^3dv\right)^{2/3}\left(\int(R-\bar{R})^6dv\right)^{1/3}}\\
                                                    & \leq{{C}{K}\left(\int(R-\bar{R})^6dv\right)^{1/2}+\frac{C}{K}\left(\int|E|^3dv\right)^2}
\end{split}
\end{equation}
The $C$ in first line just depends on volume. For the second line, we may take $K$ to be a small multiple of the Yamabe constant and absorb the first term of (\ref{R2E2 integral}) by the Soblev inequality.
The second term of (\ref{R2E2 integral}) is bounded by a constant multiple of $\int|E|^3dv$ from (\ref{E3 small}) and the smallness assumption of $\epsilon_0$.
\begin{equation}
\begin{split}
\int|R^2-\bar{R}^2|(R-\bar{R})^2dv & =\int|R+\bar{R}||R-\bar{R}|^3dv \\
                                                              & \leq{\bar{R}\int(R-\bar{R})^3dv+\int|R||R-\bar{R}|^3dv} \\
                                                              & \leq{2\bar{R}\int(R-\bar{R})^3dv+\int|R-\bar{R}|^4dv} \\
                                                              & \leq{2\bar{R}\int(R-\bar{R})^3dv+\left(\int(R-\bar{R})^2dv\right)^{1/2}\left(\int|R-\bar{R}|^6dv\right)^{1/2}}
\end{split}
\end{equation}
The last term can be absorbed by Sobolev inequality from the smallness assumption of $\epsilon_0$.
\begin{equation}
\begin{split}
\left|\bar{R}^2-\frac{\int{R^2}dv}{vol}\right|\int(R-\bar{R})^2dv &=\frac{1}{vol}\left(\int(R-\bar{R})^2dv\right)^2\\
                                                                                                             & \leq{C}\left(\int(R-\bar{R})^2dv\right)^{1/2}\left(\int(R-\bar{R})^6dv\right)^{1/2}
\end{split}
\end{equation}
This term can be absorbed by Sobolev inequality from the smallness assumption of $\epsilon_0$.

Combining all these estimates for $|R-\bar{R}|$, we derive
\begin{equation}\label{R3R6 final}
\frac{d}{dt}\int|R-\bar{R}|^3dv \leq {-c_2Y}\left(\int|R-\bar{R}|^6dv\right)^{1/2}+c_1{\bar{R}}\int|{R-\bar{R}}|^3dv+C\int|E|^3dv
\end{equation}

Now we combine (\ref{E3E6})(\ref{W3W6})(\ref{F3F6})(\ref{R3R6 final}) to derive \ref{evolution G_3}.
\end{proof}
\begin{lemma}\label{L3}
With the modified Yamabe metric chosen as initial metric, under the Ricci flow, we have
\[\sup_{\widetilde{T}/2\leq{t}\leq{\widetilde{T}}}\int{G}_3(t)dv_{g(t)}\leq{C\epsilon_0^{3/2}}.\]
\end{lemma}
\begin{proof}
The proof is fundamentally the same as that of Lemma \ref{E3}. The only difference is to replace $|E|^k$ by $G_k(t)$ since we have evolution equations of same type as is shown in
Proposition \ref{evolution of G2} and Proposition (\ref{evolution of G3}).
\end{proof}



To derive the $L^\infty$-boundedness, we shall apply the following result established by Deane Yang in \cite{Y}.

\begin{lemma}\label{Yang}
Assume that with respect to the metric $g=g(t)$, $0\leq{t}\leq{T}$, the following Sobolev inequality holds:
\begin{equation}\label{Sobolev Inequality}
\left(\int|\varphi|^{\frac{2n}{n-2}}dv\right)^{\frac{n-2}{n}}\leq{C_S\left[\int|\nabla\varphi|^2dv+\int\varphi^2dv\right]},\varphi\in{W^{1,2}(M^n)}.
\end{equation}
Also, let $b\geq0$ on $M^n\times[0,T]$ satisfy
\begin{equation}
\frac{\partial}{\partial{t}}dv\leq{bdv}.
\end{equation}
Let $q>n$, and suppose $u\geq0$ is a function on $M^n\times[0,T]$ satisfying
\begin{equation}
\frac{\partial{u}}{\partial{t}}\leq\Delta{u}+bu,
\end{equation}
and that
\begin{equation}\label{b L^q/2}
\sup_{0\leq{t}\leq{T}}|b|_{L^{q/2}}\leq\beta.
\end{equation}
Given $p_0>1$, there exists a constant $C=C(n,q,p_0,C_S,\beta)$ such that for $0\leq{t}\leq{T}$,
\begin{equation}\label{Yang's boundedness}
|u(t,\cdot)|_{\infty}\leq{C}e^{Ct}t^{-\frac{n}{2p_0}}|u(0,\cdot)|_{p_0}.
\end{equation}
Moreover, given $p\geq{p_0}>1$, the following inequality holds for $0\leq{t}\leq{T}$:
\begin{equation}\label{Yang's evolution}
\frac{d}{dt}\int{u^p}dv+\int\left|\nabla\left(u^{p/2}\right)\right|^2dv\leq{Cp^{\frac{2n}{q-n}}}\int{u^p}dv
\end{equation}
where $C=C(n,q,p_0,C_S)$.
\end{lemma}

\begin{lemma}\label{Linfty}
With the modified Yamabe metric chosen as the initial metric, we have
\[\sup_{3\widetilde{T}/4\leq{t}\leq{\widetilde{T}}}\{|E|+|R-\bar{R}|+||W^-||+|F^+_-|\}\leq{C}\epsilon_0.\]
\end{lemma}
\begin{proof}
Apply Lemma \ref{Yang} to $u=G_2$, $q=6>n=4$ and $p_0=3/2$ on $t\in[\widetilde{T}/2,\widetilde{T}]$. Condition (\ref{b L^q/2}) is satisfied by Lemma \ref{L3}. Hence, we can prove the desired estimate.
\end{proof}
Now recall the Bernstein-Bando-Shi estimate (see for example Chapter 7 of \cite{CK}).
\begin{lemma}\label{Cinfty}
Let $(M^4,g(t))$ be a solution to the Ricci flow. For every $m\in\mathbb{N}$, there exists a constant $C_m$depending only on $m$ such that  if
\[\sup_{x\in{M}}|Rm(x,t)|_{g(t)}\leq{K},\quad t\in\left[0,\frac{1}{K}\right],\]
then
\[\sup_{x\in{M}}|\nabla^mRm(x,t)|_{g(t)}\leq{\frac{C_mK}{t^{m/2}}},\quad t\in\left[0,\frac{1}{K}\right],\]
\end{lemma}

Now we are at the position to prove Theorem A.
\begin{proof}[Proof of Theorem A]
We argue by contradiction. Suppose there is a sequence of manifolds $(M_j,g_j)$ satisfying $\beta(M_j,[g_j])<4(1+\epsilon_j)$ with $\epsilon_j\to0$ and each of them is \emph{not} diffeomorphic to standard $\mathbb{CP}^2$.
For each conformal class $[g_j]$, we choose the modified Yamabe metric $(g_j)_{G}$  as initial metric and evolve the metric along Ricci flow. Then Lemma \ref{Linfty} and Lemma \ref{Cinfty} will imply that there is a time $\widetilde{T}$
such that the curvatures of $g_j(\widetilde{T})$ are uniformly bounded in $C^{\infty}$-norm and the Sobolev constants are also uniformly bounded. The convergence theory \cite{CGT} established by Cheeger, Gromov and Taylor then shows that there is a subsequence of
$\{(M_j,g_j(\widetilde{T}))\}$ which converges smoothly to a manifold $(M_\infty,g_\infty)$. As $\epsilon_j\to0$, we obtain that $(M_\infty,g_\infty)$ satisfies
\[\int_{M_\infty}||W||^2dv_\infty=\int_{M_\infty}\sigma_2dv_\infty\]
Note that we also have $b_1(M_\infty)=0$, $b_2^+(M_\infty)=1$ and $b_2^-(M_\infty)=0$. Hence, by Chern-Gauss-Bonnet and signature formula, we can easily derive that $(M_\infty,g_\infty)$ is self-dual Einstein. The same argument in Lemma \ref{CP2} will show that
$(M_\infty,g_\infty)$ is conformal equivalent to $(\mathbb{CP}^2,g_{FS})$.
Since the convergence is smooth, we thereby obtain that $(M_j,g_j)$ must be diffeomorphic to $\mathbb{CP}^2$ with standard differentiable structure when $j$ is sufficiently large. This is clearly a contradiction to our assumption. Hence, we have proved the theorem.
\end{proof}

\vskip.1in

\begin{proof}[Proof of Theorem B]
To prove Theorem B, suppose $M^4$ is oriented with $b_2^{+}(M^4) > 0$.  If $\beta(M^4) = 4$, then by definition we can find a metric $g$ with
\begin{align*}
\beta(M^4,[g]) < 4(1 + \epsilon/2),
\end{align*}
where $\epsilon > 0$ is from Theorem A.  From Theorem A we conclude that $M^4$ is diffeomorphic to $\mathbb{CP}^2$.  In addition, if $g$ is a metric on $\mathbb{CP}^2$ for which
$\beta(M^4,[g]) = 4$, then taking $\epsilon = 0$ in Lemma \ref{L2pinchLemma} we see that $g$ is self-dual.  It follows, for example, from \cite{Poon86} that $(M^4,[g])$ is conformally equivalent to $(\mathbb{CP}^2,g_{FS})$.
\end{proof}

\bibliographystyle{amsplain}

\end{document}